\documentclass[12pt]{amsart}
\usepackage{amssymb,amsmath,amsthm,mathrsfs,verbatim}
\newcommand{\sm}[4]{\left(\begin{smallmatrix}#1&#2\\ #3&#4 \end{smallmatrix}
\right)}
\newtheorem{theorem}{Theorem}
\newtheorem{lemma}[theorem]{Lemma}
\newtheorem{corollary}[theorem]{Corollary}

\newtheorem{proposition}[theorem]{Proposition}

\theoremstyle{remark}

\newtheorem*{example}{Example}
\numberwithin{theorem}{section} \numberwithin{equation}{section}

\newcommand{\C}{\mathbb{C}}

\newcommand{\Q}{\mathbb{Q}}

\newcommand{\Z}{\mathbb{Z}}
\newcommand{\N}{\mathbb{N}}
\newcommand{\SL}{{\text {\rm SL}}}

\def\H{\mathbb{H}}

\newcommand{\F}{{}_2F_1\left(\frac{1}{2},\frac{1}{2};1,16t\right)}
\newcommand{\Fg}{{}_2F_1\left(\frac{1}{2},\frac{1}{2};1,t\right)}
\newcommand{\Fh}{{}_2F_1\left(a,b;c,t\right)}
\newcommand{\Fl}{{}_2F_1\left(\frac{1}{2},\frac{1}{2};1,16l(\tau)\right)}

\begin{document}
\title[Modular forms, hypergeometric functions and congruences] {Modular forms, hypergeometric functions and congruences}

\author{Matija Kazalicki}
\address{Department of Mathematics, University of Zagreb,
Croatia, Zagreb, Bijeni\v{c}ka cesta 30} \email{mkazal@math.hr}


\begin{abstract}
Using the theory of Stienstra and Beukers \cite{BS},  we prove various elementary congruences for the numbers 
$$\displaystyle \sum_{\substack{i_1,i_2\ldots i_k \ge 0\\ i_1+i_2+\cdots i_k=n}}\binom{2i_1}{i_1}^2\binom{2i_2}{i_2}^2\cdots\binom{2i_k}{i_k}^2, \quad \textrm{ for } k,n\in \N.$$ 
To obtain that, we study the arithmetic properties of Fourier coefficients of certain (weakly holomorphic) modular forms.
\end{abstract}
\maketitle
\section{Introduction and statement of results}

Consider the family of elliptic curves given by the Legendre equation
$$y^2=x(x-1)(x-t),\quad t\in \C,$$
whose period integrals 
$$\Omega_1(t)=\int_{t}^{1}\frac{dx}{\sqrt{x(x-1)(x-t)}}, \quad \Omega_2(t)=\int_1^\infty \frac{dx}{\sqrt{x(x-1)(x-t)}},$$
are solutions of the differential equation of Picard-Fuch type
$$t(t-1)\Omega''(t)+(2t-1)\Omega'(t)+\frac{1}{4}\Omega(t)=0.$$
One finds that $\Omega_2(t)=\pi \, \Fg,$
where $$\Fh=\sum_{n=0}^\infty \frac{(a)_n (b)_n}{(c)_n n!}t^n$$ is the Gauss hypergeometric function.
This further gives identity
\begin{equation}\label{eq}
\theta(\tau)=\Fl=\sum_{n=0}^\infty \binom{2n}{n}^2 l(\tau)^n.
\end{equation}
Throughout the paper, $\tau \in \H$, $q=e^{\pi i \tau}$, $\theta(\tau)=\left(\sum_{n\in \Z}q^{n^2}\right)^2$ is the classical weight $1$ theta series,
and $l(\tau)=q-8q^2+44q^3-192q^4+\cdots$ is the normalized elliptic modular lambda function (hauptmoduln for $\Gamma(2)$). For more details see \cite{Periods}.

In this paper, we identify certain (weakly holomorphic) modular forms  $$f(\tau)=P(l(\tau))\theta^k(\tau) \frac{dl(\tau)}{d\tau},$$ for some $P(t)\in \Z[t]$, and $k \in \N$, whose Fourier coefficients are easily understood in terms of elementary arithmetic (e.g. a spliting behavior of primes in the quadratic extensions). Using identity (\ref{eq}), we exploit the relation (via formal group theory) between the coefficients of the power series $P(t)\F^k$, and Fourier coefficients of $f(\tau)$ to prove some elementary congruences for the numbers
$$\displaystyle A_k(n)=\sum_{\substack{i_1,i_2\ldots i_k \ge 0\\ i_1+i_2+\cdots i_k=n}}\binom{2i_1}{i_1}^2\binom{2i_2}{i_2}^2\cdots\binom{2i_k}{i_k}^2, \quad \textrm{ for } k,n\in \N.$$

Beukers and Stienstra \cite{BS} invented this approach to study congruence properties of Apery numbers
$$B(n)=\sum_{k=0}^\infty \binom{n+k}{k}\binom{n}{k}^2.$$ Using the formal Brauer group of some elliptic K3-surface, they proved that for all primes $p$, and $m,r\in \N$ with $m$ odd
$$B\left(\frac{mp^r-1}{2}\right)-a(p) B\left(\frac{mp^{r-1}-1}{2}\right)+(-1)^{\frac{p-1}{2}}p^2 B\left(\frac{mp^{r-2}-1}{2}\right) \equiv 0 \pmod{p^r},$$
where $\eta^6(4\tau)=\sum_{n=1}^\infty a(n)q^{2n}$.

Many authors have subsequently studied arithmetic properties of $B(n)'s$ and discovered similar three term congruence relations for other Apery like numbers. For related work see  \cite{Beukers, Verrill2,mccarthy, osburn,Verrill1, Zagier}. 

In contrast to these result, the novelity of this paper is that we use families of modular forms, as well as the weakly holomorphic modular form to extract information about congruence properties of numbers $A_k(n)$. In particular, even though Foureier coefficients of weakly holomorphic modular forms do not satisfy three term relation satisfied by coefficients of Hecke eigenforms, they sometimes satisfy three term congruence relation of Atkin and Swinnerton-Dyer type (see \cite{Matija}), which then give rise to the three term congruence relations satisfied by corresponding Apery like numbers.

Let $\Delta$ be the free subgroup of $\SL_2(\Z)$ generated by
the matrices $A=\sm 1 2 0 1$ and $B=\sm 1 0 2 1$. Note that
$\Gamma(2)=\{\pm I\}\Delta$, and that $\Gamma_1(4)={\sm {1/2} 0 0 1}\Delta{\sm {1/2} 0 0 1}^{-1}$. 
Divisors of $\theta(\tau)$, $l(\tau)$, and $1-16l(\tau)$ are suported at cusps. For an integer $k$, we denote by $M_k(\Delta)$ and $S_k(\Delta)$ the spaces of modular forms and cusp forms of weight $k$ for group $\Delta$.

First we identify some (weakly holomorphic) modular forms for $\Delta$ that have ``simple'' Fourier coefficients.

\begin{theorem}\label{thm:1}\mbox{}

For $n\in \N$, let
$$h_{n}(\tau)=\theta(\tau)^{6n+1}(1-16l(\tau))^{\lfloor\frac{n+1}{2}\rfloor}l(\tau)^{2n}=\sum_{m=1}^\infty a_n(m)q^m.$$ 
Then $h_{n}(\tau)\in S_{6n+1}(\Delta)$, and for a prime $p\equiv (-1)^{n+1} \pmod{4}$, we have that $a_n(p)=0$.
\end{theorem} 
 
\begin{theorem}\label{thm:2}\mbox{}
\begin{itemize}
\item[a)] The modular form $$f_1(\tau)=l(\tau)(1-16l(\tau))\theta(\tau)^5=\sum_{n=1}^\infty b_1(n)q^n\in M_5(\Delta)$$ is the cusp form with complex multiplication by $\Q(i)$. In particular, for a prime $p>2$ we have 
\begin{equation*} 
b_1(p)= \left\{
\begin{array}{ll}
2x^4-12x^2 y^2+2y^4 & \textrm{ if }p\equiv 1 \pmod{4}, \textrm{ and } p=x^2+y^2,\\
0 & \textrm{ if } p \equiv 3 \pmod{4}. \\
\end{array} \right.
\end{equation*}

\item[b)] For an integer $n>2$, the Fourier coefficients $c_1(n)$ of the weakly holomorphic modular form $g_1(\tau)=l(\tau)^2(1-16l(\tau))^2\theta(\tau)^5=\sum_{n=2}^\infty c_1(n)q^n$ satisfy the following congruence relation
$$b_1(n)\equiv 108 c_1(n) \pmod{n^3}.$$
\item[c)] For integer $n>1$, let
$$f_{n}(\tau)=\theta(\tau)^{6n-6}(1-16l(\tau))^{\lfloor\frac{n-1}{2}\rfloor}l(\tau)^{2n-2}f_1(\tau)=\sum_{m=1}^\infty b_{n}(m)q^m.$$ 
Then for a prime $p\equiv (-1)^{n} \pmod{4}$, we have that $b_{n}(p)= 0$.
 
\end{itemize}
\end{theorem}

\begin{corollary}\label{cor:1} Let $p>3$ be a prime, and $r \in \N$. If $p \equiv 1 \pmod 4$, let $x$ and $y$ be integers such that $p=x^2+y^2$. Denote by $D_3(n)=A_3(n-1)-16A_3(n-2).$ Then the following congruences hold
\begin{equation}\label{eq:3} A_3(mp^r-1)-b_1(p)A_3(mp^{r-1}-1)+\left(\frac{-1}{p}\right)p^4A_3(mp^{r-2}-1)\equiv 0 \pmod{p^r},
\end{equation}

\begin{equation}\label{eq:4}
D_3(mp^r-1)-b_1(p)D_3(mp^{r-1}-1)+\left(\frac{-1}{p}\right)p^4D_3(mp^{r-2}-1)\equiv 0 \pmod{p^{r-\frac{\left(\frac{-1}{p}\right)+1}{2}}}.
\end{equation}

In particular, we have

\begin{equation}\label{eq:1}
   A_3(p-1)\equiv \left\{
	\begin{array}{ll}
		16x^4 \pmod{p}& \textrm{ if }p\equiv 1 \pmod{4},\\
	0\pmod{p} & \textrm{ if } p \equiv 3 \pmod{4}. \\
	\end{array} \right.
\end{equation}

	\begin{equation}\label{eq:2} 
   D_3(p-1)\equiv \left\{
	\begin{array}{ll}
		\frac{4}{27}x^4\pmod{p} & \textrm{ if }p\equiv 1 \pmod{4},\\
	0 \pmod{p} & \textrm{ if } p \equiv 3 \pmod{4}. \\
	\end{array} \right.
\end{equation}
\end{corollary}

For integers $n\ge 1$ and $m\ge 0$, define the sequences $B_n(m)$ and $C_n(m)$ by the following identities
$$(1-16t)^{\lfloor \frac{n-1}{2}\rfloor}t^{2n-1} \F^{6n-1}=\sum_{m=0}^\infty B_n(m) t^m,$$
$$(1-16t)^{\lfloor \frac{n-1}{2} \rfloor}t^{2n-2}\F^{6n-3}=\sum_{m=0}^\infty C_n(m)t^m.$$

\begin{corollary}\label{cor:2}
For $n\in \N$ and a prime $p>2$, we have that $B_n(p-1)\equiv 0 \pmod{p}$, if $p \equiv (-1)^{n+1} \pmod{4}$. Moreover, $C_n(p-1)\equiv 0 \pmod{p}$ if $p \equiv (-1)^n \pmod{4}$.
\end{corollary}

\begin{example} Let $p>2$ be a prime. Consider the coefficients of $\F^2$. The corresponding modular form $$l(\tau)(1-16l(\tau))\Fl^4=\sum_{k=1}^\infty (-1)^k \left( \sigma_3(k/2)-\sigma_3(k)\right)q^k$$ is an Eisenstein series whose $p^{th}$ Fourier coefficient is $\equiv 1 \bmod{p}$ (if $k$ is odd, we define $\sigma_3(k/2)$ to be $0$). Hence Lemma \ref{cor:Beukers} implies
$$\sum_{k=0}^{p-1} \binom{2k}{k}^2\binom{2(p-1-k)}{p-1-k}^2\equiv 1 \pmod{p},$$
or equivalently $\displaystyle\binom{p-1}{\frac{p-1}{2}}^4\equiv 1 \pmod{p}.$
\end{example}


\section{Acknowledgements}

I would like to thank Max Planck Institute for Mathematics in Bonn for the excelent working environment and the financial support that they provided for me during my stay in March of 2012.

\section{Modular forms for $\Delta$}
\subsection{Preliminaries}\label{prelim}
Let 
\begin{equation} \label{Dedekind}
\eta(\tau)=q^{\frac{1}{12}}\prod_{n=1}^\infty (1-q^{2n})
\end{equation}
be a Dedekind eta function (recall $q=e^{\pi i \tau}$).
As we mentioned in the introduction

\begin{equation}\label{eq:l}
l(\tau)=\frac{\eta(2\tau)^{16}\eta(\tau/2)^8}{\eta(\tau)^{24}} \quad\textrm{ and}
\end{equation}

\begin{equation}\label{eq:1-16l}
1-16l(\tau)=\frac{\eta(\tau/2)^{16}\eta(2\tau)^8}{\eta(\tau)^{24}}
\end{equation} are modular functions for $\Delta$. They are holomorphic
on $\H$, and $l(\tau)\ne 0,1/16$ for all $\tau
\in \H$.

Modular curve $X(2)$ has three cusps: $0$, $1$, and $\infty$. As functions on $X(2)$, $l(\tau)$ and $1-16l(\tau)$ have simples poles at $\infty$ and zeros of order $1$ at cusps $0$ and $1$ respectively. 

It is well known that 

\begin{equation}\label{eq:theta}
\theta(\tau)=\frac{\eta(\tau)^{10}}{\eta(\tau/2)^4\eta(2\tau)^4}
\end{equation}
is a modular form of weight 1 for $\Delta$. It has a zero at cusp $1$ of order $1/2$ (cusp $1$ is irregular). 

We will later need the following dimension formula for the spaces of cusp forms.
Let $\Gamma$ be a finite index subgroup of $\SL_2(\Z)$ of genus $g$ such that $-I\notin \Gamma$. For $k$ odd, \cite[Theorem 2.25]{Shimura} gives the following formula for the dimension of $S_k(\Gamma)$
\begin{equation}\label{eqn:dim}
\dim S_k(\Gamma) = (k-1)(g-1)+\frac{1}{2}(k-2)r_1+\frac{1}{2}(k-1)r_2+\sum_{i=1}^{j}\frac{e_i-1}{2e_i},
\end{equation}
where $r_1$ is the number of regular cusp, $r_2$ is the number of irregular cusps, and $e_i's$ are the orders of elliptic points. Since $\Delta$ has no elliptic points (it is a free group), it follows that $\dim S_5(\Delta)=1$. 

\subsection{Identities}
In this subsection we list some technical facts and identities which will be used later in the proofs of the theorems. The proofs of these statements are straightforward, so detail are ommited. 
\begin{lemma}\label{lemma:1}
	Let $$\displaystyle\Psi(\tau)=(1-16l(\tau))^{1/2}l^2(\tau)\theta(\tau)^6=q^2+\sum_{m=2}^\infty a(m)q^{2m}.$$ Then, 
	$\Psi(\tau) \in S_6(\Gamma_0(4))$ is a newform.
	Moreover $a(2)=0$, 
	hence the coefficients of the Fourier expansion of $\Psi(\tau)$ in $q$ are supported at integers congruent to $2 \bmod{4}$. 
\end{lemma}

\begin{lemma}\label{lemma:2}
The following identity holds
$$\displaystyle\frac{\eta(\tau/2+1/2)}{\eta(\tau/2)}=\frac{\eta(\tau)^3}{\eta(\tau/2)^2\eta(2\tau)}.$$
\end{lemma}
\begin{proof}
	It follows from the product formula (\ref{Dedekind}) for Dedekind eta function.
\end{proof}

The following lemma follows from the previous lemma and equations (\ref{eq:1-16l}) and (\ref{eq:theta}).  

\begin{lemma}\label{lemma:3}
We have that
$$\frac{\theta(\tau+1)}{\theta(\tau)}=(1-16l(\tau))^{1/2}.$$
\end{lemma}
\noindent
Denote by $D=\frac{1}{\pi i}\frac{d}{d\tau}=q\frac{d}{dq}$.
We will need the following curious identities.
\begin{lemma}\label{lemma:4}
The following equality holds
$$-\frac{1}{12}D^4\left(\frac{1}{\theta(\tau)^3}\right)=(1-16l(\tau))l(\tau)\theta(\tau)^5-108(1-16l(\tau))^2l(\tau)^2\theta(\tau)^5.$$
\end{lemma}

\begin{proof}
By Bol's theorem \cite{Bol}, the lefthand side of the equality is a weakly holomorphic modular form of weight 5 for group $\Delta$. One checks that  the initial Fourier coefficients of the both sides of the equality agree, hence the lemma follows.  
\end{proof}

\begin{lemma}\label{lemma:5} The following identity holds
$$D(l(\tau))=l(\tau)(1-16l(\tau))\theta(\tau)^2.$$
\end{lemma}
\begin{proof}
By Bol's theorem \cite{Bol}, $D(l(\tau))$ is a weakly holomorphic modular form of weight 2. The initial Fourier coefficients of both forms agree, hence the claim follows.
\end{proof}


\section{Proofs of the theorems}

\subsection{Proof of Theorem \ref{thm:1} and Theorem \ref{thm:2}}

\begin{proof}[Proof of Theorem \ref{thm:1}]
It is easy to check that $h_n(\tau)$ vanishes at cusps (see Section \ref{prelim}). 
Denote by $\nu(\tau)=\theta(\tau)^{12}(1-16l(\tau))l(\tau)^4$.
If $n=2k$ is even, then $h_n(\tau)=\theta(\tau)\nu(\tau)^k$. From Lemma \ref{lemma:1}, it follows that $\nu(\tau)$ has Fourier coefficients supported at integers that are congruent $0\bmod{4}$. 
Since $\theta(\tau)=\sum_{i=0}^\infty r_2(i)q^i$ has a property that $r_2(4j-1)=0$ for $j\in \N$, the claim follows.

If $n=2k+1$ is odd, then $h_n(\tau)=h_1(\tau)\nu(\tau)^k$. Therefore, it is enough to prove the statement of the theorem for $h_1(\tau)$.
Since, by Lemma \ref{lemma:1}, the Fourier coefficients of $\theta(\tau)^6(1-16l(\tau))^{1/2}l(\tau)^2$ are supported at integers that are congruent to $2\bmod{4}$, it follows that the coefficients of $\theta(\tau)^7(1-16l(\tau))^{1/2}l^2(\tau)$ are supported at integers congruent to $3\bmod{4}$. The same is true for the coefficients of $h_1(\tau)$ since 
Lemma \ref{lemma:1} and Lemma \ref{lemma:3} imply that $$h_1(\tau)= \theta(\tau + 1)^7(1-16l(\tau+1))^{1/2}l(\tau+1)^2.$$ 
\end{proof}

\begin{proof}[Proof of Theorem \ref{thm:2}]
a) It is easy to check that $f_1(\tau)$ is a cusp form (see Section \ref{prelim}). 
The space of cusp forms $S_5(\Delta)$ is one dimensional (see equation (\ref{eqn:dim})). Since it contains a CM modular form with the stated properties, the claim follows.\newline
b) It follows from Lemma \ref{lemma:4}.
\newline
c) Same as the proof of Theorem \ref{thm:1}.
\end{proof}

\subsection{Proof of Corollary \ref{cor:1} and Corollary \ref{cor:2}}

We need the following result of Beukers \cite{Beukers}.

\begin{proposition}[Beukers]\label{proposition:Beukers}
Let $p$ be a prime and 
$$\omega(t)=\sum_{n=1}^\infty b_nt^{n-1}dt$$
a differential form with $b_n\in \Z_p$. Let $t(q)=\sum_{n=1}^\infty A_nq^n$,$A_n\in \Z_p$, and suppose
$$\omega(t(q))=\sum_{n=1}^\infty c_n q^{n-1}dq.$$
Suppose there exist $\alpha_p,\beta_p\in \Z_p$ with $p|\beta_p$ such that 
$$b_{mp^r}-\alpha_p b_{mp^{r-1}}+\beta_pb_{mp^{r-2}}\equiv 0 \pmod{p^r}, \quad \forall m,r\in \N.$$
Then
$$c_{mp^r}-\alpha_p c_{mp^{r-1}}+\beta_pc_{mp^{r-2}}\equiv 0 \pmod{p^r}, \quad \forall m,r\in \N.$$
Moreover, if $A_1$ is $p$-adic unit then the second congruence implies the first.
\end{proposition}

\begin{lemma}\label{cor:Beukers}
Define a differential form
$$\omega(t)=\sum_{n=1}^\infty b_nt^{n-1}dt,$$ where $b_n\in \Z_p$. Let $t(q)=\sum_{n=1}^\infty A_nq^n$, with $A_n\in \Z_p$ and $A_1 = 1$. Define
$$\omega(t(q))=\sum_{n=1}^\infty c_n q^{n-1}dq.$$
For a prime $p$, we have that $b_p \equiv c_p \pmod{p}$.
\end{lemma}
\begin{proof} We have that
$$\sum_{n=1}^\infty c_n q^{n-1} = \frac{d}{dq}\left(\sum_{n=1}^\infty \frac{b_n}{n}t(q)^n\right).$$ Hence $c_p \bmod{p}$ is equal to the $(p-1)^{th}$ coefficient of $\displaystyle\frac{d}{dq}\left(\frac{b_p}{p}t(q)^p\right)$, which is $b_p$.
\end{proof}

\begin{proof}[Proof of Corollary \ref{cor:1}]
Since $f_1(\tau)$ is an eigenform for the Hecke operator $T_p$, the assumptions of Proposition \ref{proposition:Beukers} are satisfied for $c_n = b_1(n)$, $b_n=A_3(n-1)$,  $\alpha_p=b_1(p)$, $\beta_p=\left(\frac{-1}{p}\right) p^4$, and $t(q)=16l(\tau)$. The formula (\ref{eq:3})  follow from identity
$$D(l(\tau))\Fl^3=f_1(\tau),$$
(which is a consequence of equation (\ref{eq}) and Lemma \ref{lemma:5}).

To prove (\ref{eq:4}) we use Theorem \ref{thm:2} b) to establish three term congruence relation from Proposition \ref{proposition:Beukers} between coefficients $c_1(np^r)$ (because there is one satisfied by $b_1(np^r)$'s). Note that assumptions of Proposition \ref{proposition:Beukers} are satisfied if we take,  when $p\equiv 3 \pmod{4}$,  $c_n = c_1(n)$, $b_n=D_3(n-1)$,  $\alpha_p=b_1(p)$, $\beta_p=\left(\frac{-1}{p}\right) p^4$, and $t(q)=16l(\tau)$ (since then $b(p)=0$).

If $p\equiv 1 \pmod{4}$, we need to take $c_n = pc_1(n)$, $b_n=pD_3(n-1)$,  $\alpha_p=b_1(p)$, $\beta_p=\left(\frac{-1}{p}\right) p^4$, and $t(q)=16l(\tau)$, hence we get a congruence relation weaker for one power of $p$.

Formulas (\ref{eq:1}) and (\ref{eq:2}) follow from Lemma \ref{cor:Beukers} with the choice of parameters as above.

\end{proof}

\begin{proof}[Proof of Corollary \ref{cor:2}]
We apply Lemma \ref{cor:Beukers} to cusp forms $f_{n}(\tau)$ and $h_{n}(\tau)$, where we choose $t(q)$ to be the inverse of $l(q)$ under composition. The claim follows from $a_n(p)=b_n(p)=0$.
\end{proof}

\bibliographystyle{siam}
\bibliography{bibl}

\begin{thebibliography}{10}

\bibitem{Beukers}
{\sc F.~Beukers}, {\em Another congruence for the {A}p\'ery numbers}, J. Number
  Theory, 25 (1987), pp.~201--210.

\bibitem{Bol}
{\sc G.~Bol}, {\em Invarianten linearer differentialgleichungen}, Abh. Math.
  Sem. Univ. Hamburg, 16 (1949), pp.~1--28.

\bibitem{Verrill2}
{\sc F.~Jarvis and H.~A. Verrill}, {\em Supercongruences for the
  {C}atalan-{L}arcombe-{F}rench numbers}, Ramanujan J., 22 (2010),
  pp.~171--186.

\bibitem{Matija}
{\sc M.~Kazalicki and A.~J. Scholl}, {\em Modular forms, de {R}ham cohomology
  and congruences}, preprint.

\bibitem{Periods}
{\sc M.~Kontsevich and D.~Zagier}, {\em Periods}, in Mathematics
  unlimited---2001 and beyond, Springer, Berlin, 2001, pp.~771--808.

\bibitem{mccarthy}
{\sc D.~McCarthy, R.~Osburn, and B.~Sahu}, {\em Arithmetic properties for
  {A}p\'ery-like numbers}, preprint.

\bibitem{osburn}
{\sc R.~Osburn and B.~Sahu}, {\em Congruences via modular forms}, Proc. Amer.
  Math. Soc., 139 (2011), pp.~2375--2381.

\bibitem{Shimura}
{\sc G.~Shimura}, {\em Introduction to the arithmetic theory of automorphic
  functions}, Publications of the Mathematical Society of Japan, No. 11.
  Iwanami Shoten, Publishers, Tokyo, 1971.
\newblock Kan{\^o} Memorial Lectures, No. 1.

\bibitem{BS}
{\sc J.~Stienstra and F.~Beukers}, {\em On the {P}icard-{F}uchs equation and
  the formal {B}rauer group of certain elliptic {$K3$}-surfaces}, Math. Ann.,
  271 (1985), pp.~269--304.

\bibitem{Verrill1}
{\sc H.~A. Verrill}, {\em Congruences related to modular forms}, Int. J. Number
  Theory, 6 (2010), pp.~1367--1390.

\bibitem{Zagier}
{\sc D.~Zagier}, {\em Integral solutions of {A}p\'ery-like recurrence
  equations}, in Groups and symmetries, vol.~47 of CRM Proc. Lecture Notes,
  Amer. Math. Soc., Providence, RI, 2009, pp.~349--366.

\end{thebibliography}

\end{document}